\documentclass[a4paper,12pt]{article}
\usepackage{a4wide}
\usepackage{amsmath}
\usepackage{amssymb}
\usepackage{amsthm}
\usepackage{latexsym}
\usepackage{graphicx}
\usepackage[english]{babel}
\usepackage{makeidx}

\newtheorem{obs} [subsection]{Remark}
\newtheorem{exm} [subsection]{Example}

\newtheorem{prop}[subsection]{Proposition}
\newtheorem{conj}[subsection]{Conjecture}
\newtheorem{teor}[subsection]{Theorem}
\newtheorem{lema}[subsection]{Lemma}
\newtheorem{cor} [subsection]{Corollary}
\newcommand{\Zng}{$\mathbb Z^n$-graded $S$-module}

\def\sdepth{\operatorname{sdepth}}
\def\depth{\operatorname{depth}}
\def\supp{\operatorname{supp}}

\def\lcm{\operatorname{lcm}}

\begin{document}
\selectlanguage{english}
\frenchspacing

\begin{center}
\textbf{On the Stanley depth of powers of some classes of monomial ideals}
\vspace{10pt}

\large
Mircea Cimpoea\c s
\end{center}
\normalsize

\begin{abstract}
Given arbitrary monomial ideals $I$ and $J$ in polynomial rings $A$ and $B$ over a field $K$, we investigate the Stanley depth of powers of the sum $I+J$, and their quotient rings, in $A\otimes_K B$ in terms of those of $I$ and $J$. Our results can be used to study the asymptotic behavior of the Stanley depth of powers of a monomial ideal. For instance, we solved the case of monomial complete intersection.


\end{abstract}

\section*{Introduction}

Let $K$ be a field and $S=K[x_1,\ldots,x_n]$ the polynomial ring over $K$.
Let $M$ be a \Zng. A \emph{Stanley decomposition} of $M$ is a direct sum $\mathcal D: M = \bigoplus_{i=1}^rm_i K[Z_i]$ as a $\mathbb Z^n$-graded $K$-vector space, where $m_i\in M$ is homogeneous with respect to $\mathbb Z^n$-grading, $Z_i\subset\{x_1,\ldots,x_n\}$ such that $m_i K[Z_i] = \{um_i:\; u\in K[Z_i] \}\subset M$ is a free $K[Z_i]$-submodule of $M$. We define $\sdepth(\mathcal D)=\min_{i=1,\ldots,r} |Z_i|$ and $\sdepth_S(M)=\max\{\sdepth(\mathcal D)|\;\mathcal D$ is a Stanley decomposition of $M\}$. The number $\sdepth_S(M)$ is called the \emph{Stanley depth} of $M$. Herzog, Vladoiu and Zheng show in \cite{hvz} that $\sdepth_S(M)$ can be computed in a finite number of steps if $M=I/J$, where $J\subset I\subset S$ are monomial ideals. In \cite{rin}, Rinaldo give a computer implementation for this algorithm, in the computer algebra system $\mathtt{CoCoA}$ \cite{cocoa}. For a friendly introduction in the thematic of Stanley depth, we refer the reader \cite{her}.

Let $I$ and $J$ be two monomial ideal in polynomial rings $A$ and $B$ over a field $K$. In \cite{trung}, H.\ H.\ Ha, N.\ V.\ Trung and T.\ N.\ Trung studied the behavour of the $\depth$ for powers of $I+J$, in terms of those of $I$ and $J$. Our aim is to do a similar task for the $\sdepth$. However, since there are no homological methods for computing the $\sdepth$, our results are weaker.
In Proposition $2.4$ we show that $\sdepth((I+J)^n/(I+J)^{n+1}) \geq \min_{i+j=n} \{ \sdepth_A(I^i/I^{i+1}) + \sdepth_B(J^j/J^{j+1}) \}$, for all $n\geq 0$. 

Also, $\sdepth((I+J)^n) \geq \min\{ \sdepth(I^n), \sdepth(I^{n-1}J/I^{n}J), \ldots, \sdepth(J^n/IJ^n)\}$ for any $n\geq 1$, see Proposition $2.6$. There are no general results regarding the asymptotic behavior of the numerical functions $n \mapsto \sdepth_B(B/J^n)$, $n \mapsto \sdepth_B(J^n)$ and $n \mapsto \sdepth_B(J^n/J^{n+1})$, like in the case of $\depth$, see for instance \cite{hibi}. If $J\subset B$ is a monomial complete intersection ideal, we prove that $\sdepth(R/(I+J)^n)\geq \min_{i\leq n}\{\sdepth_A(A/I^i) + \dim(B/J)\}$, see Theorem $2.11$. 

In the general case, is difficult to estimate $\sdepth(R/(I+J)^n)$. A partial result for $n=2$ is given in Proposition $2.9$.
If $J\subset B$ is a monomial complete intersection, we prove that $\sdepth_B(B/J^n) = \sdepth_B(J^n/J^{n+1}) = \dim(B/J)$, and we give bounds for $\sdepth_B(J^n)$, see Theorem $2.15$.

\footnotetext[1]{The support from grant ID-PCE-2011-1023 of Romanian Ministry of Education, Research and Innovation is gratefully acknowledged.}

\newpage
\section{Preliminaries}

We consider $I\subset S$ a monomial ideal. We say that $I$, respectively $S/I$, satisfies the \emph{Stanley inequality}, if $\sdepth(I)\geq \depth(I)$, respectively $\sdepth(S/I)\geq \sdepth(S/I)$. Stanley \cite{stan} conjectured in fact that for any ideal $I\subset S$, both $I$ and $S/I$ satisfy these inequalities. This conjecture proved to be false in the case of $S/I$, see \cite{duval}.

In \cite{lukas}, L.\ Katth\"an proposed the following conjecture, which is a weaker form of the Stanley conjecture, and it is open.

\begin{conj} (See \cite[Theorem 3.1]{lukas})

$(1)$ $\sdepth(S/I)\geq \depth(S/I)-1$.

$(2)$ $\sdepth(I)\geq \depth(I)$.
\end{conj}

Another conjecture, stated by Herzog in \cite{her}, is the following.

\begin{conj}
$\sdepth(I)\geq \sdepth(S/I)+1$.
\end{conj}

B.\ Ichim proved that this conjecture is true for $n\leq 5$, see \cite[Theorem 1.2]{ichim}. Also, this result is true in some other particular cases, for example, when $I$ is a monomial complete intersection, see \cite[Theorem 2.4]{mirc}.

Now, we recall the well known Depth Lemma, see for instance \cite[Lemma 1.3.9]{real}. 

\begin{lema}(Depth Lemma)
If $0 \rightarrow U \rightarrow M \rightarrow N \rightarrow 0$ is a short exact sequence of modules over a local ring $S$, or a Noetherian graded ring with $S_0$ local, then

a) $\depth M \geq \min\{\depth N,\depth U\}$.

b) $\depth U \geq \min\{\depth M,\depth N +1 \}$.

c) $\depth N\geq \min\{\depth U - 1,\depth M\}$.
\end{lema}

In \cite{asia}, Asia Rauf proved the analog of Lemma $1.3(a)$ for $\sdepth$:

\begin{lema}
Let $0 \rightarrow U \rightarrow M \rightarrow N \rightarrow 0$ be a short exact sequence of $\mathbb Z^n$-graded $S$-modules. Then:
\[ \sdepth(M) \geq \min\{\sdepth(U),\sdepth(N) \}. \]
\end{lema}

Let $R=K[y_1,\ldots,y_m]$ be a polynomial ring. Let $M$ be a graded $S$-module and $N$ be a graded $R$-module. It is well known, see for instance \cite[Lemma 2.5]{trung}, that: $$\depth(M\otimes_K N)=\depth(M) + \depth(N).$$ 

W.\ Bruns, C.\ Krattenthaler, and J.\ Uliczka proved in \cite{bruns} a similar result for $\sdepth$: 

\begin{lema}\cite[Proposition 2.10]{bruns}
\[ \sdepth(M\otimes_K N) \geq \sdepth(M) + \sdepth(N) \]
\end{lema}

As a consequence of Lemma $1.5$, if $\sdepth(M)\geq\depth(M)$ and $\sdepth(N)\geq\depth(N)$, then $\sdepth(M\otimes_K N)\geq \depth(M\otimes_K N)$.

\newpage
\section{Main results}

Let $r,s\geq 1$ be two integers, and let $A=K[x_1,\ldots,x_r]$ and $B=K[x_{r+1},\ldots,x_{r+s}]$ be polynomial rings over a field $K$, and let $R=K[x_1,\ldots,x_{r+s}]$. Let $I\subset A$ and $J\subset B$ be nonzero proper monomial ideals. In order to simplify the notations, we will denote also by $I$ and $J$, their extensions in $R$. We recall the following results of L.\ T.\ Hoa and N.\ D.\ Ham.

\begin{lema} \cite[Lemma 1.1 and 3.2]{hoa}
(a) $IJ = I\cap J$.

(b) $\depth(R/IJ)=\depth_A(A/I)+\depth_B(B/J)+1$.
\end{lema}

As a direct consequence of Lemma $2.1$ and \cite[Theorem 1.2]{mirci}, we get the following result:

\begin{prop}
$(1)$ $\sdepth(IJ)\geq \sdepth_A(I)+\sdepth_B(J)$.

$(2)$ $\sdepth(R/I)\geq \sdepth(R/IJ) \geq \min \{\sdepth(R/I), \sdepth_B(B/J) + \sdepth_A(I) \}$

$(3)$ $\sdepth(R/J)\geq \sdepth(R/IJ) \geq \min \{\sdepth(R/J), \sdepth_A(A/I) + \sdepth_A(J) \}$

$(4)$ If $I$ and $J$ satisfy the Stanley inequality, then $IJ$ also satisfy the Stanley inequality.

$(5)$ If $1.2$ holds for $I$ or $J$, and $A/I$ and $B/J$ satisfy the Stanley inequality, then $R/IJ$ satisfies the Stanley inequality.

$(6)$ If $I$ and $J$ satisfy $1.1(2)$ and $A/I$ and $B/J$ satisfy $1.1(1)$, then $R/IJ$ satisfies $1.1(1)$.
\end{prop}

H.\ H.\ Ha, N.\ V.\ Trung and T.\ N.\ Trung proved in \cite{trung} the following result.

\begin{prop}
$(I+J)^n/(I+J)^{n+1} = \bigoplus_{i+j=n} (I^i/I^{i+1})\otimes_K (J^j/J^{j+1})$, for all $n\geq 0$. 
\end{prop}

As a direct consequence of Lemma $1.4$, Lemma $1.5$ and Proposition $2.3$, we get:

\begin{prop} For all $n\geq 0$, we have:
$$\sdepth((I+J)^n/(I+J)^{n+1}) \geq \min_{i+j=n} \{ \sdepth_A(I^i/I^{i+1}) + \sdepth_B(J^j/J^{j+1}) \}.\Box$$
\end{prop}

For a monomial $v\in S$, we denote $\supp(v)=\{x_j\;:\;x_j|v\}$ the \emph{support} of $v$.

\begin{prop}
Let $J\subset B$ be a complete intersection monomial ideal. It holds that \linebreak $\sdepth(J^n/J^{n+1})=\dim(B/J)$, for all $n\geq 0$.
\end{prop}

\begin{proof}
We denote $G(J)=\{ v_1,\ldots,v_t\}$ the set of minimal monomial generators of $J$.
We use induction on $t\geq 1$. If $t=1$, then we set $v=v_1$ and we get $J^n/J^{n+1} = (v^n)/(v^{n+1}) = v^n(S/(v))$. Therefore there is nothing to prove. If $t>1$, we let $B'=K[x_i\;:\;i>r,\;x_i\nmid v_t]$ and $B''=K[x_i\;:\;i>r,
\;x_i| v_t]$. Since $J$ is a complete intersection we get $v_i\in B'$, for $i<t$. 

We write $J=J'+J''$ where $J'=(v_1,\ldots,v_{t-1})\subset B'$ and $J''=(v_t)\subset B''$. 
Using the induction hypothesis we get $\sdepth(J'^n/J'^{n+1}) = \dim(B'/J')$. By Proposition $2.4$, it follows that 
$\sdepth_B(J^n/J^{n+1})\geq \dim(B'/J')+\dim(B''/J'') = \dim(B/J)$.

Let $uK[Z]\subset J^n/J^{n+1}$ be a Stanley space. We claim that for each $1\leq j \leq t$, there exists some $r+1 \leq i_j \leq r+s$ such that $x_{i_j}\in \supp(v_j)\setminus Z$. Indeed, if $\supp(v_j) \subset Z$, then $uv_j^{n+1}\in uK[Z]$ and $uv_j^{n+1}\in J^{n+1}$, a contradiction. Since the monomials $v_j$'s have disjoint supports, it follows that $|Z|\leq s-t=\dim(B/J)$ and therefore $\sdepth_B(B/J^n)\leq \dim(B/J)$, as required.
\end{proof}

Let $Q_i:=\sum_{j=0}^i I^{n-j}J^j$, for $0\leq i\leq n$. Note that there is a chain of ideals:
\[ I^n = Q_0 \subset Q_1 \subset \cdots \subset Q_n=(I+J)^n. \]
According to \cite[Lemma 2.2]{trung}, we have $Q_i/Q_{i-1}\cong I^{n-i}J^i/I^{n-i+1}J^i$, for all $1\leq i\leq n$.
Thus, using repeatedly Lemma $1.4$, we get the following result.

\begin{prop}
For all $n\geq 0$, we have:
$$\sdepth((I+J)^n) \geq \min\{ \sdepth(I^n), \sdepth(I^{n-1}J/I^{n}J), \ldots, \sdepth(J^n/IJ^n)\}. \Box$$
\end{prop}

\begin{prop} For all $n\geq 0$, we have:
\[ \sdepth((I+J)^{n})\geq \min\{ \sdepth((I+J)^{n+1}), \sdepth((I+J)^n/(I+J)^{n+1}) \}. \]
\[ \sdepth(R/(I+J)^{n+1})\geq \min\{ \sdepth(R/(I+J)^n), \sdepth((I+J)^n/(I+J)^{n+1}) \}. \]
\end{prop}

\begin{proof}
\emph{We consider the short exact sequences: $$0 \rightarrow (I+J)^{n+1} \rightarrow (I+J)^{n} \rightarrow (I+J)^{n}/(I+J)^{n+1} \rightarrow 0\;\;and $$ 
$$0 \rightarrow (I+J)^{n}/ (I+J)^{n+1} \rightarrow R/(I+J)^{n+1} \rightarrow R/(I+J)^{n} \rightarrow 0.$$ }
By Lemma $1.4$, we get the required inequalities.
\end{proof}

\begin{obs}
\emph{If we consider the short exact sequence $0 \rightarrow Q_i/Q_{i-1} \rightarrow R/Q_{i-1} \rightarrow R/Q_i \rightarrow 0$, then, by Lemma $1.4$, we get $\sdepth(R/Q_{i-1})\geq \min\{\sdepth(R/Q_{i}),\sdepth(Q_i/Q_{i-1})\}$, for all $1\leq i \leq n$.}

\emph{Also, from the short exact sequence $0 \rightarrow Q_i/Q_{i-1} \rightarrow R/I^{n-i+1}J^i \rightarrow R/I^{n-i}J^i \rightarrow 0 $
, by Lemma $1.4$, we get $\sdepth(R/I^{n-i+1}J^i)\geq \min\{\sdepth(R/I^{n-i}J^i),\sdepth(Q_i/Q_{i-1})\}$, for all $1\leq i \leq n$.}
\end{obs}

\begin{prop}
$(1)$ If $\sdepth(R/(I^2+IJ))<\sdepth_A(A/I)+\sdepth_B(J^2)$, then: $$\sdepth(R/(I+J)^2)\leq \sdepth_A(A/I)+\sdepth_B(J^2).$$

$(2)$ If $\sdepth(R/I^2)<\sdepth_A(I/I^2)+\sdepth_B(J)$, then: $$\sdepth(R/(I^2+IJ))\leq \sdepth(R/I^2).\;\;\emph{Moreover}\;\; \sdepth_A(A/I)\leq \sdepth_A(A/I^2).$$
\end{prop}

\begin{proof}
We consider the following short exact sequences:
\[(i)\;\; 0 \longrightarrow J^2/IJ^2 \longrightarrow R/(I^2+IJ) \longrightarrow R/(I+J)^2 \longrightarrow 0. \]
\[(ii)\;\; 0 \longrightarrow IJ/I^2 \longrightarrow R/I^2 \longrightarrow R/(I^2+IJ) \longrightarrow 0, \]

Note that $IJ/I^2J \cong (I/I^2)\otimes_K J$ and $J^2/IJ^2 \cong (A/I)\otimes_K J^2$. Therefore, by Lemma $1.5$,
$\sdepth(IJ/I^2J)\geq \sdepth_A(I/I^2)+\sdepth_B(J)$ and $\sdepth(J^2/IJ^2)\geq \sdepth_A(A/I)+\sdepth_B(J^2)$.

By $(i)$ and Lemma $1.4$, it follows that $\sdepth(R/(I^2+IJ))\geq \min\{ \sdepth(J^2/IJ^2), \linebreak
\sdepth(R/(I+J)^2) \} \geq  \min\{\sdepth_A(A/I)+\sdepth_B(J^2), \sdepth(R/(I+J)^2) \},$
hence we get $(1)$. By $(ii)$ and Lemma $1.4$, it follows that {\small
$$\sdepth(R/I^2)=\sdepth_A(A/I^2)+s\geq  \min\{  \sdepth_A(I/I^2)+\sdepth_B(J), \sdepth(R/(I^2+IJ))\}.$$}
On the other hand, if $\sdepth_A(A/I)$ $>\sdepth_A(A/I^2)$, then, by Lemma $1.4$ and the short exact sequence $0 \rightarrow I/I^2 \rightarrow A/I^2 \rightarrow A/I \rightarrow 0$, it follows that $\sdepth_A(I/I^2)\leq \sdepth_A(A/I^2)$, which contradicts the hypothesis of (2). Thus we are done. 
\end{proof}

In the following, we present an example of a monomial ideal $I\subset S$ with $\sdepth_A(A/I)<\sdepth_A(A/I^2)$, firstly given by J.\ Herzog and T.\ Hibi in \cite{hibi}, in the frame of $\depth$. 

\begin{exm}
\emph{We consider the ideal {\small $I=(x_1^6, x_1^5x_2, x_1x_2^5, x_2^6, x_1^4x_2^4x_3,x_1^4x_2^4x_4,x_1^4x_5^2x_6^3,x_2^4x_6^2x_5^3)$}
 in $A=K[x_1,\ldots,x_6]$. We have $\depth_A(A/I)=0$, $\depth_A(A/I^2)=1$, $\depth_A(A/I^3)=0$ and \linebreak $\depth_A(A/I^4)=\depth_A(A/I^5)=2$.}

\emph{Using $\mathtt{CoCoA}$ \cite{cocoa}, we get $\sdepth_A(A/I)=0$, $\sdepth_A(A/I^2)=1$ and $\sdepth_A(I)=3$.}
\end{exm}

\begin{teor}
If $J\subset B$ is a complete intersection monomial ideal, then for all $n\geq 1$:
$(1)$ $\depth(R/(I+J)^n) = \min_{i\leq n}\{\depth_A(A/I^i) + \dim(B/J)\}$.

$(2)$ $r+\dim(B/J) \geq \sdepth(R/(I+J)^n)\geq \min_{i\leq n}\{\sdepth_A(A/I^i) + \dim(B/J)\}$.

$(3)$ In particular, if $A/I^i$ satisfy the Stanley inequality (or Conjecture $1(1)$), for all $1\leq i\leq n$, then $R/(I+J)^n$ also satisfies the Stanley inequality (or Conjecture $1(1)$).

$(4)$ $\sdepth(R/(I+J)^{n+1})\leq \sdepth(R/(I+J)^n)$,  $\sdepth((I+J)^{n+1})\leq \sdepth((I+J)^n)$ and 
$\sdepth((I+J)^{n+1}/(I+J)^{n+2})\leq \sdepth((I+J)^n/(I+J)^{n+1})$, for all $n\geq 1$.

$(5)$ The numerical functions $n\mapsto \sdepth(R/(I+J)^n)$, $n\mapsto \sdepth((I+J)^n)$ and $n\mapsto \sdepth((I+J)^{n}/(I+J)^{n+1})$ are constant for $n\gg 0$.
\end{teor}

\begin{proof}
Assume $G(J)=\{v_1,\ldots,v_t\}$, where $1\leq t\leq s$.
Note that $\depth_B(B/J^n)=s-t = \dim(B/J)$. We use induction on $t\geq 1$. If $t=1$, we set $v=v_1$. Then $(I+J)^n=(I,v)^n = I^n + vI^{n-1} + \cdots + v^{n-1}I + (v^n)$. We claim that 
\[(i)\;\;
R/(I,v)^n= \bigoplus_{\alpha=0}^{n-1} \bigoplus_{v^{\alpha}|u,\;v^{\alpha+1}\nmid u} (A/I^{n-\alpha})u \otimes_K B/(v^{\alpha}).
\]
Indeed, let $w\in R \setminus (I,v)^n$ be a monomial. Let $\alpha:=\max\{j\geq 0\;:\;v^j|u\}$ and write $w=v^{\alpha}u$. Note that $\alpha<n$ and $v\nmid u$. 
Note that $u\notin I^{n-\alpha}$, otherwise $w\in I^{n-\alpha}v^n \subset (I,v)^n$, a contradiction. Thus, we proved our claim.

Note that $\sdepth_B(B/(v^n))=\depth_B(B/(v^n))=s-1$. By Lemma $1.3$ and $(i)$, it follows $(1)$, in the case $t=1$. Also, by Lemma $1.4$, Lemma $1.5$ and $(i)$, we get $(2)$, for $t=1$.

Now, assume $t>1$. Let $A'=A[x_i\;:\;x_i\nmid v_t]$, $B'=K[x_i\;:\;x_i|v_t]$ and $B''=K[x_i\;:\;i>r\;,x_i\nmid v_t]$. 
Let $J'=(v_t)\subset B'$ and $J''=(v_1,\ldots,v_{t-1})\subset B''$. We can write $I+J = I'+J'$, where $I' = I+ J''\subset A'$. By induction, we get $ \depth_{A'}(A'/I'^n) = \min_{i\leq n}\{ \depth_A(A/I^i) + \dim(B''/J'') \}$. 

On the other hand, by the first step of induction, we get $\depth(R/(I'+J')^n) = \min_{i\leq n} \{ \depth_{A'}(A'/I'^i) + \dim(B'/J') \}$. Since $\dim(B/J)=\dim(B'/J')+\dim(B''/J'')$, we get $(1)$. 

Similarly, we get $\sdepth(R/(I+J)^n)\geq \min_{i\leq n}\{\sdepth_A(A/I^i) + \dim(B/J)\}$.

If $uK[Z]\subset R/(I+J)^n$ is a Stanley space, as in the proof of Proposition $2.5$, for each $1\leq j \leq t$ there exists some $r+1\leq i_j\leq r+s$ such that $x_{i_j}|v_j$ and $x_{i_j}\notin Z$. Since the monomials $v_j$`s have disjoint supports, it follows that $|Z|\leq r+s-t$ and therefore $\sdepth(R/(I+J)^n)\leq r+s-t = r+\dim(B/J)$, as required.

$(3)$ Follows immediately from $(2)$.

$(4)$ Using induction on the number of generators of $J$, it is enough to assume that $J=(u)$ is principal. Since $u$ is regular on $R/IR$, it follows that $((I,u)^{n+1}:u)=(I,u)^n$. Therefore, by \cite[Proposition 2.11]{cimp}, we are done.

$(5)$ Follows immediately from $(4)$.
\end{proof}

A direct consequence of Theorem $2.11$ is the following Corollary.

\begin{cor}
If $J\subset B$ is a complete intersection monomial ideal minimally generated by $t$ monomials, 
then $\sdepth_B(B/J^n)= \depth_B(B/J^n)=s-t$, for all $n\geq 1$. $\Box$
\end{cor}

\begin{cor}
Let $L\subset R$ be a monomial ideal with $G(L)=\{v_1,\ldots,v_m,v\}$. Assume there exists a monomial $w\in A$ such that $gcd(v,v_i)=w$ for all $i\leq m$ and $v/w \in B$. Then, $\depth(R/L^n)\leq \depth(R/L^{n+1})$ for all $n\geq 1$.
\end{cor}

\begin{proof}
Let $n\geq 1$. Note that $L^n=w^n(L:w)^n$ and therefore $\depth(R/L^n)=\depth(R/(L:w)^n)$. Thus, we can assume that $w=1$. We write $L=I+J$, where $I=(v_1,\ldots,v_m)\subset A$ and $J=(v)\subset B$. We apply Theorem $2.11(1)$ and we get the required conclusion.
\end{proof}

A \emph{semilattice} $L$ is a partial ordered set $(L,\leq)$ such that, for any $P,Q \in L$, there is a unique least upper bound
$P \vee Q$ called the \emph{join} of $P$ and $Q$. Let $L,L'$ be two semilattices. A \emph{join-preserving} map $\varphi:L \rightarrow L'$ is a map with $\varphi(P \vee Q) = \varphi(P) \vee \varphi(Q)$.

Let $S$ be a polynomial ring in $n\geq 1$ indeterminates.
The \emph{lcm-semilattice} $L_G$ of a finite set $G\subset S$ of monomials is defined
as the set of all monomials that can be obtained as the least common multiple ($\lcm$) of
some non-empty subset of $G$, ordered by divisibility. If $I\subset S$ is a monomial ideal, we define the lcm-lattice of $I$, 
$L_I:=L_{G(I)}$. Note that $u \vee v = \lcm(u,v)$, for any monomials $u,v\in L_I$.
See \cite{ichim1} for further details. 

\begin{prop}
Let $J\subset B$ be a monomial complete intersection ideal with $|G(J)|=t\leq s$. Then $s-t+1 \leq \sdepth(J^k)\leq s-t+
\left\lceil \frac{t}{k+1}\right\rceil$, for all $k\geq 1$.

In particular, if $k\geq t-1$, then $\sdepth(J^k)=s-t+1$.
\end{prop}

\begin{proof}
Let $\mathbf m=(x_1,\ldots,x_t)\subset S:=K[x_1,\ldots,x_t]$. Suppose that $J=(v_1,\ldots,v_t)$ and let $k\geq 1$.
Since $v_j$'s are monomials with disjoint supports, it follows that $L_{I^k}$ contains all the monomials of the form 
$v_1^{k_1}\cdots v_t^{k_t}$ with $k_i\geq 0$ and $1\leq k_1+\cdots+k_t\leq k$. Also, one can easily see that $\varphi:L_{\mathbf m^k} \rightarrow L_{I^k}$, defined by $\varphi(x_1^{k_1}\cdots x_t^{k_t}):=v_1^{k_1}\cdots v_t^{k_t}$, is bijective and join preserving.
Thus, according to \cite[Theorem 4.5]{ichim1}, we have $\sdepth(I^k)=r-t+\sdepth(\mathbf m^k)$. On the other hand, according to
\cite[Theorem 2.2]{cim}, $1\leq \sdepth(\mathbf m^k)\leq \left\lceil \frac{t}{k+1}\right\rceil$. 
\end{proof}

According to \cite[Theorem 1.2]{hibi}, if $J\subset B$ is a homogeneous ideal, then there exists a constant $c\geq 0$, such that $\depth_B(B/J^n) = \depth_B(J^n/J^{n+1}) = c$, for $n\gg 0$. One may ask if a similar result is true in the case of $\sdepth$. As a consequence of Proposition $2.5$, Corollary $2.12$ and Proposition $2.14$, we get the following theorem, which give a particular case when the answer is positive.

\begin{teor}
If $J\subset B$ is a complete intersection monomial ideal, then:

$(1)$ $\sdepth_B(B/J^n) =\sdepth_B(J^n/J^{n+1}) = \dim(B/J)$, for all $n\geq 0$.

$(2)$ $\sdepth_B(J^n)=\dim(B/J)+1$, for all $n\geq s-\dim(B/J)-1$. $\Box$
\end{teor}

\vspace{2mm} \noindent {\footnotesize
\begin{minipage}[b]{15cm}
Mircea Cimpoea\c s, Simion Stoilow Institute of Mathematics, Research unit 5, P.O.Box 1-764,\\
Bucharest 014700, Romania, E-mail: mircea.cimpoeas@imar.ro
\end{minipage}}

\end{document}